\def\mymode{i}  

\documentclass{amsart}

\usepackage{begnac}

\renewcommand{\concat}{}
\DeclareMathOperator{\Int}{Int}

\title{Polish topometric groups}

\author{Ita\"i Ben Yaacov}

\address{Ita\"i \textsc{Ben Yaacov} \\
  Universit\'e de Lyon \\
  CNRS UMR 5208 \\
  Universit\'e Lyon 1 \\
  Institut Camille Jordan \\
  43 blvd. du 11 novembre 1918\\
  F-69622 Villeurbanne Cedex \\
  France}

\urladdr{\url{http://math.univ-lyon1.fr/~begnac/}}

\author{Alexander Berenstein}

\address{Alexander Berenstein \\
  Universidad de los Andes \\
  Departamento de Matemáticas \\
  Carrera 1 \# 18A-10, Bogotá \\
  Colombia}

\urladdr{\url{http://matematicas.uniandes.edu.co/~aberenst}}

\author{Julien Melleray}

\address{Julien \textsc{Melleray} \\
  Universit\'e de Lyon \\
  CNRS UMR 5208 \\
  Universit\'e Lyon 1 \\
  Institut Camille Jordan \\
  43 blvd. du 11 novembre 1918\\
  F-69622 Villeurbanne Cedex \\
  France}

\urladdr{\url{http://math.univ-lyon1.fr/~melleray/}}

\begin{document}

\thanks{Work on this project was facilitated by ANR chaire d'excellence junior THEMODMET (ANR-06-CEXC-007), a visit of the third author at the Erwin Schr\"odinger Institute in Vienna, and the ECOS Nord program (action ECOS Nord C10M01).
  The first author is also supported by the Institut Universitaire de France}

\if\mymode i
\svnInfo $Id: TopMetGen.tex 1292 2011-11-18 15:56:49Z begnac $
\thanks{\textit{Revision} {\svnInfoRevision} \textit{of} \today}
\fi

\keywords{ample generics, topometric group, automorphism group}
\subjclass[2010]{03E15,37B05}

\begin{abstract}
  We define and study the notion of \emph{ample metric generics} for a Polish topological group, which is a weakening of the notion of ample generics introduced by Kechris and Rosendal in \cite{Kechris-Rosendal:Turbulence}.
  Our work is based on the concept of a \emph{Polish topometric group}, defined in this article.
  Using Kechris and Rosendal's work as a guide, we explore consequences of ample metric generics (or, more generally, ample generics for Polish topometric groups).
  Then we provide examples of Polish groups with ample metric generics, such as the isometry group $\Iso(\bU_1)$ of the bounded Urysohn space, the unitary group ${\mathcal U}(\ell_2)$ of a separable Hilbert space, and the automorphism group $\Aut([0,1],\lambda)$ of the Lebesgue measure algebra on $[0,1]$.
  We deduce from this and earlier work of Kittrell and Tsankov that this last group has the automatic continuity property, i.e., any morphism from $\Aut([0,1],\lambda)$ into a separable topological group is continuous.
\end{abstract}
\maketitle

\section{Introduction}
This paper presents a technique to study ``large'' Polish groups\footnote{for definitions of most concepts discussed in the abstract and in this Introduction, see \fref{sec:Prelim} below.}.
These groups, which typically appear as automorphism groups of highly homogeneous structures, may have some surprising properties; an example could be the small index property, which states that any subgroup of index strictly less than the continuum is open.
Another interesting possible attribute of such groups is the Bergman property: whenever $G$ acts by isometries on a metric space, all the orbits are bounded.
Or one could think of extreme amenability (all continuous actions on compact sets have fixed points), the Steinhaus property (which implies that any morphism from $G$ to a separable group $H$ is continuous)\ldots the list could go on and on.

In their paper \cite{Kechris-Rosendal:Turbulence}, Kechris and Rosendal continued work initiated by Hodges, Hodkinson, Lascar and Shelah in \cite{Hodges-Hodkinson-Lascar-Shelah:SmallIndex} and introduced the notion of \emph{ample generics} for a Polish group $G$: $G$ has ample generics if for any $n$ there exists a tuple $\bar g = (g_0,\ldots,g_{n-1}) \in G$ whose diagonal conjugacy class
\begin{gather*}
  \bar g^G = \{(kg_0k^{-1},\ldots,kg_{n-1}k^{-1}) \colon k \in G \}
\end{gather*}
is co-meagre in $G^n$.
Kechris and Rosendal showed that this notion provides a unified approach to some of the properties discussed above, as any Polish group with ample generics must have the small index property and the Steinhaus property (among others).
Ample generics is also closely related to the Bergman property.
Finally, having ample generics is a strong condition, and Polish groups with ample generics are relatively rare.

Frustratingly, Polish groups with ample generics are even scarcer than one would hope: at the moment all known examples are subgroups of ${\mathfrak S}_{\infty}$, the permutation group on the set of integers.
It is a standard fact that the closed subgroups of ${\mathfrak S}_{\infty}$ are exactly (up to a homeomorphic isomorphism) the automorphism groups of countable logical structures, and as such are far from covering the whole variety of Polish groups (for example, they are totally disconnected).
General Polish groups are, on the other hand, exactly the automorphism groups of Polish metric structures, in the sense of \emph{continuous logic}, a generalisation of first order logic introduced by A.\ Usvyatsov and the first author in \cite{BenYaacov-Usvyatsov:CFO}.
We thus find ourselves doing some kind of transition from a discrete setting to a metric one, and it may be worthwhile to see how model theory has handled the same transition.

In metric model theory (see \cite{BenYaacov-Berenstein-Henson-Usvyatsov:NewtonMS} for a general survey) it is a recurrent practice to take some object which, in the discrete setting, would be purely topological (most commonly, a type space), add some natural distance which \emph{refines} the topology, and then adapt definitions to take this additional distance into account (since the topology alone need no longer provide enough information).
This goes back to Iovino's definition of $\lambda$-stability in Banach spaces \cite{Iovino:StableBanach}, as well as to Henson's version of the Ryll-Nardzewski Theorem for metric structures (unpublished by its author, but see \cite{BenYaacov-Usvyatsov:dFiniteness}).
The formalism of topometric spaces which we use here is introduced in \cite{BenYaacov-Usvyatsov:CFO,BenYaacov:TopometricSpacesAndPerturbations}, where it is used to define a generalisation of the Cantor-Bendixson analysis, for the purpose of defining local stability and $\aleph_0$-stability in metric structures.
One sanity check which such modifications must pass is that for a discrete distance the modified definitions should coincide with the classical ones.

In the situation at hand, the automorphism group of a metric structure $G = \Aut(\cM)$ is equipped with the topology $\tau$ of point-wise convergence, as in the discrete case, and at the same time can be equipped with the distance $\partial$ of \emph{uniform} convergence (which is discrete for a discrete structure, so passes unnoticed).
The topology $\tau$ is a group topology, and is Polish whenever $\cM$ is Polish; the distance $\partial$ is bi-invariant (so in particular, induces a group topology), refines $\tau$, and is not in general separable.
If a Polish group $(G,\tau)$ is not given as an automorphism group, it still admits a coarsest bi-invariant distance refining the topology $\tau$ of $G$.
One is thus led to study triplets  $(G,\tau,\partial)$ where $(G,\tau)$ is a Polish group and $\partial$ is a bi-invariant distance which refines $\tau$ (and satisfies an additional compatibility condition, see \fref{sec:Prelim}); such triplets will be called \emph{Polish topometric groups}.
Even in the case of a pure Polish group $(G,\tau)$, taking the induced bi-invariant distance $\partial_u$ into account will allow us to retrieve information about $(G,\tau)$.

We say that a Polish topometric group $(G,\tau,\partial)$ has \emph{ample generics} if it admits diagonal conjugacy classes whose $\partial$-closures are $\tau$-co-meagre.
If $(G,\tau,\partial_u)$ has ample generics we say that $(G,\tau)$ has \emph{ample metric generics}.
The reader will notice that this passes our sanity check requirement.

First, this relaxation of the definition of ample generics allows new examples.
Indeed, in \fref{sec:Examples} we provide examples of groups $G$ without ample generics (in fact, in which all diagonal conjugacy classes are known to be meagre), which do have ample metric generics.
Such groups include the isometry group of the universal Urysohn space, the group of measure-preserving automorphisms of the unit interval, and the unitary group of a complex Hilbert space.

Second, the relaxed definition allows results similar to those in \cite{Kechris-Rosendal:Turbulence}, discussed in \fref{sec:Consequences}.
Our main automatic continuity result is the following.

\begin{thm*}
  Let $(G,\tau,\partial)$ be a Polish topometric group with ample generics, $H$ a separable topological group and $\varphi \colon G \to H$ a morphism such that $\varphi \colon (G,\partial) \to H$ is continuous.
  Then $\varphi \colon (G,\tau) \to H$ is continuous.
\end{thm*}

This theorem, along with work of Kittrell and Tsankov \cite{Kittrell-Tsankov:TopologicalPropertiesOfFullGroups}, and the fact that the group $\Aut([0,1],\lambda)$ of measure preserving transformations of the unit interval has ample metric generics, yields the following result.

\begin{thm*}
  Let $\varphi \colon \Aut([0,1],\lambda) \to H$ be a morphism, with $H$ a separable topological group.
  Then $\varphi$ is continuous when $\Aut([0,1],\lambda)$ is endowed with its usual Polish topology.
\end{thm*}

From this and \cite{Glasner:AutMuRoelckePrecompact} one may deduce the following result.

\begin{thm*}
The group $\Aut([0,1],\lambda)$ has a unique  topology which is Hausdorff, second-countable, and compatible with its group structure (namely, its usual Polish group topology).
\end{thm*}

We should note here that Kallman \cite{Kallman:UniquenessForGroups} had already proved that $\Aut([0,1],\lambda)$ admits a unique Polish group topology compatible with its algebraic structure.

The structure of most of our arguments is very similar to those of Kechris and Rosendal; in some proofs in Sections \ref{sec:TopometricAmpleGenerics} and \ref{sec:Consequences} of the paper, whenever one replaces $\varepsilon$ by $0$ one obtains an argument that appears verbatim in \cite{Kechris-Rosendal:Turbulence}.
For the convenience of the reader, we still present the proofs in full detail.

\section{Preliminaries}
\label{sec:Prelim}

\subsection{Polish spaces and groups}
\label{sec:Polish}

Recall that a \emph{Polish space} is a separable topological space $X$ which is completely metrisable, i.e., admits a compatible complete distance.
Some times we want the distance on $X$ to be fixed: by a \emph{Polish metric space} we mean a complete separable metric space $(X,d)$.

Whenever $(X,d)$ is a metric space, one can consider its isometry group $\Iso(X,d)$, the group of distance-preserving bijections from $X$ onto itself, equipped with the (group) topology of point-wise convergence.
When $(X,d)$ is a Polish metric space, $\Iso(X,d)$ is a \emph{Polish group}, i.e., a topological group which is Polish as a topological space.
Note that, by the Birkhoff-Kakutani theorem, any Polish group admits a compatible left-invariant distance (which need not be complete).

A major tool in the study of Polish spaces is the Baire category theorem.
Recall that a subset $A$ of a topological space is \emph{meagre} if it is a countable union of nowhere-dense sets.
It is \emph{co-meagre} if its complement is meagre, or equivalently, if it contains a countable intersection of dense open sets.
The Baire category theorem asserts that, in a Polish space, any co-meagre set is dense.
This notion is particularly useful in Polish groups, because of Pettis' theorem, which asserts that whenever $G$ is a Polish group and $A$ is a subset of $G$ that is co-meagre in some non empty open set then $A {\cdot} A^{-1}$ contains a neighbourhood of $1_G$.
We shall make frequent use of this theorem below.
Let us point out two well-known consequences of Pettis' theorem: first, if $H$ is a subgroup of a Polish group $G$ which is co-meagre in some non empty open subset of $G$, then $H$ is open in $G$.
Second, if $\varphi \colon G \to H$ is a continuous, bijective morphism between two Polish groups, then $\varphi$ is a homeomorphism.

We need to recall some other classical notions of descriptive set theory.
If $X$ is a Polish space and $A \subseteq X$, one says that $A$ is \emph{Baire-measurable} if there exists an open subset $O$  of $X$ such that $A \Delta O$ is meagre.
These sets form a $\sigma$-algebra that contains the Borel subsets of $X$.
Whenever a Baire-measurable set is non meagre it must be co-meagre in some non empty open set, and this fact is frequently used below, often in conjunction with Pettis' theorem.
Another important property of Baire-measurable sets is the Kuratowski-Ulam Theorem.

\begin{fct}[Kuratowski-Ulam Theorem]
  Let $X$ and $Y$ be Polish spaces and let $A \subseteq X \times Y$ be Baire-measurable.
  For $x \in X$ let $A_x = \{y\colon (x,y) \in A\}$.
  Then $A$ is co-meagre in $X \times Y$ if and only if $A_x$ is co-meagre in $Y$ for all $x$ in a co-meagre set, and in any case $A_x$ is Baire-measurable for all $x$ in a co-meagre set.
\end{fct}

Given an ambient space $X$, the notation $\forall^* x \ P(x)$ means that the set $\{x \in X \colon P(x) \}$ is co-meagre in $X$.
Then the first assertion of the Theorem can be restated as
\begin{gather*}
  \forall^* (x,y) \  (x,y) \in A \quad \Longleftrightarrow \quad \forall^* x\  \forall^* y \  (x,y) \in A.
\end{gather*}

The $\sigma$-algebra of Baire-measurable sets of a Polish space $X$ contains (if $X$ is uncountable) sets that are non Borel.
In particular, any \emph{analytic} subset of $X$ is Baire-measurable.
Recall that a subset $A$ of $X$ is analytic if there exists a Polish space $Y$ and a Borel subset $B$ of $X\times Y$ such that $A$ is equal to the projection of $B$ on the first coordinate.
In what follows, we shall use two facts about analytic sets, besides their Baire-measurability: the intersection of countably many analytic subsets of a Polish space $X$ is still analytic; and whenever $X,Y$ are Polish metric spaces, $A \subseteq X$ is analytic and $f \colon X \to Y$ is continuous, the set $f(A)$ is an analytic subset of $Y$.

For information about Polish spaces and groups, and about descriptive set theory, the reader is invited to consult \cite{Gao:InvariantDescriptiveSetTheory}, \cite{Kechris:Classical} and the bibliographical references therein.

\subsection{Topometric groups}

We define the main objects of study of this paper.

\begin{dfn}
  We say that a triplet $(X,\tau,\partial)$ is a \emph{topometric space} if $X$ is a set, $\tau$ is a topology on $X$ and $\partial$ is a distance such that the topology induced by $\partial$ refines $\tau$ in such a way that $\partial$ is $\tau$-lower semi-continuous, i.e., the set $\{(x,y) \colon \partial(x,y) \le r\}$ is closed for all $r$.

  A \emph{topometric group} is a triplet $(G,\tau,\partial)$ which is a topometric space and is such that $(G,\tau)$ is a topological group and $\partial$ is bi-invariant.
\end{dfn}

(Topometric spaces were originally introduced in \cite{BenYaacov-Usvyatsov:CFO,BenYaacov:TopometricSpacesAndPerturbations}.)

\begin{exm}
  \label{exm:UniformDistance}
  Let $(G,\tau)$ be a metrisable group and let $d_L$ be a compatible left-invariant distance on $G$.
  We then define
  \begin{gather*}
    \partial_u(g,h) = \sup_{k \in G} d_L(gk,hk).
  \end{gather*}
  One can check that $\partial_u$ is bi-invariant, refines $\tau$, and is coarsest such (up to uniform equivalence).
  Moreover, $(G,\tau,\partial_u)$ is a topometric group, and if $G$ is Polish then $\partial_u$ is complete (even though $d_L$ need not be).
\end{exm}

By analogy with distances, one could say that a uniformity on a group is \emph{bi-invariant} if it is generated by bi-invariant entourages.
Then every topological group admits a coarsest bi-invariant uniformity refining its topology, and if the group is metrisable then so is the corresponding bi-invariant uniformity, being given by $\partial_u$ as above.

By a convenient abuse of notation, we shall tend to ignore the distinction between any two uniformly equivalent distances.
Specifically, whenever $(G,\tau)$ is a Polish group and $\partial$ is a distance that generates the coarsest bi-invariant uniformity refining $\tau$, we shall write $\partial=\partial_u$, even though of course there are many different distances that generate this uniformity.

\begin{rmk}
  It is important to point out that some distances we consider (and most bi-invariant distances, such as $\partial_u$) need \emph{not} be separable.
  Such distances will be denoted throughout by $\partial$, to emphasise the fact that these are not the distances one is used to considering when dealing with Polish spaces and groups.
  For example, if $(G,\tau)$ is the unitary group $U(\ell_2)$ equipped with the strong operator topology, then $\partial_u$ is (equivalent to) the non separable operator norm.
\end{rmk}

\begin{dfn}
  A \emph{topometric action} of a Polish group $(G,\tau_G)$ on a topometric space $(X,\tau,\partial)$ is a continuous group action (i.e., $(g,x) \mapsto g {\cdot} x$ is continuous) which acts by $\partial$-isometries.
  Whenever $G$ acts topometrically on $X$, we shall consider the diagonal action of $G$ on $X^n$, which is also naturally a topometric action on $X^n$ ($X^n$ being endowed with the product topology and the supremum metric).
\end{dfn}

An interesting aspect of the metric in topometric spaces is that it enables us to enlarge sets slightly: if $(X,\tau,\partial)$ is a topometric space, $A \subseteq X$ and $\varepsilon >0$ then we define
$$(A)_\varepsilon= \{x \in X \colon \partial(x,A) < \varepsilon\}.$$

In the following, we focus on Polish topometric spaces and groups, that is to say we only consider topometric spaces and groups whose topology is Polish.
In this setting, our technical assumption on $\partial$ implies in particular that it is a Borel function from $X \times X$ to $\bR$, and so it is easy to check that, whenever $A \subseteq X$ is analytic and $\varepsilon >0$, $(A)_\varepsilon$ is also analytic.
Indeed, $(A)_\varepsilon$ is equal to the projection on the first coordinate of the following subset of $X \times X$:
\begin{gather*}
  \bigl\{ (x,y) \in X \times X \colon x \in X \text{ and } y \in A \text{ and } \partial(x,y)<\varepsilon \bigr\}.
\end{gather*}
This set is analytic, and the projection of an analytic subset of $X\times X$ is an analytic subset of $X$.
Consequently, the uniform closure of any analytic subset of $X$ is again analytic.

\subsection{Metric logical structures}

Let us begin with a convention that we shall follow throughout the paper: whenever $(M,d)$ is a metric space and $n \in \bN$, we endow $M^n$ with the supremum metric, which we also denote by $d$.

\begin{dfn}
  A \emph{metric structure} $\cM$ is a complete bounded metric space $(M,d)$, along with a family $(P_i)_{i \in I}$ of \emph{continuous predicates}, i.e., uniformly continuous bounded maps $P_i\colon M^{k_i} \to \bR$, and a family $(f_j)_{j \in J}$ of \emph{functions}, i.e., uniformly continuous maps $f_j\colon M^{\ell_j} \to M$ (where $k_i,\ell_j \in \bN$).
  We always assume that the distance function $d$ is included in our list of predicates.

  The structure is said to be \emph{Polish} if the underlying metric space is, that is, if $M$ is separable.

  If $\cM$ and $\cN$ are metric structures, an \emph{isomorphism} $\varphi\colon \cM \to \cN$ is a bijection $\varphi\colon M \to N$ which preserves all the predicates (including the distance) and functions.
  If $\cM = \cN$ then it is an \emph{automorphism}.
\end{dfn}

The functions of a metric structure are in some sense superfluous, since we may always replace a function $f(\bar x)$ with the predicate $G_f(\bar x,y) = d(f(\bar x),y)$ without changing the automorphism group (a structure involving no functions is said to be \emph{relational}).
For structures which arise ``naturally'', however (e.g., some of those mentioned in \fref{sec:Examples}), it is convenient to allow functions as well as predicates.

\begin{fct}
  Let $\cM$ be a Polish metric structure with universe $(M,d)$, and denote by $G = \Aut(\cM)$ its automorphism group, endowed with the point-wise convergence topology $\tau$.
  Then $G$ is a closed subgroup of $\Iso(M,d)$ and is therefore a Polish group.

  Let $\partial_u^\cM$ denote the distance of uniform convergence on $G$, namely
  \begin{gather*}
    \partial_u^\cM(g,h) = \sup_{x \in M} d(gx,hx).
  \end{gather*}
  Then the topology $\tau$ is naturally refined by the bi-invariant distance $\partial_u^\cM$, and $(G,\tau,\partial_u^\cM)$ is a topometric group.
\end{fct}

The following is less immediate.
We state it without proof, for the sake of completeness, and shall not make any real use of it (the crucial point is that by the metric Ryll-Nardzewski Theorem, the $\aleph_0$-categoricity assumption implies that the action of $G$ on $\cM$ is approximately oligomorphic, as per \fref{dfn:ApproximatelyOligomorphicAction}).
\begin{fct}
  \label{fct:Aleph0CatUniformity}
  Assume $\cM$ is a Polish metric structure whose continuous first order theory is $\aleph_0$-categorical, and let $(G,\tau,\partial_u^\cM) = \Aut(\cM)$ as above.
  Then $\partial_u^\cM$ defines the coarsest bi-invariant uniformity refining $\tau$, i.e., $\partial_u = \partial_u^\cM$.
\end{fct}

Conversely,

\begin{dfn}
  If $A \subseteq M$ is complete and closed under functions, we may equip $A$ with the \emph{induced structure} from $\cM$ in the natural way, making it a \emph{metric sub-structure} of $\cM$.
  In the case of an arbitrary subset $A \subseteq M$, we let $\langle A \rangle$ be the closure of $A$ under functions, observing that the (complete) metric closure $\overline{\langle A \rangle}$ is again closed under functions, and is therefore a metric sub-structure.

  We say that two tuples $\bar a$ and $\bar b$ in $M^n$ have \emph{the same quantifier-free type} if there exists an isomorphism $\varphi\colon \overline{\langle \bar a \rangle} \to \overline{\langle \bar b \rangle}$ sending each $a_i$ to $b_i$, observing that if such an isomorphism exists then it is unique.
  Notice that if $\cM$ is relational, then this is the same as saying that for every predicate $P$ taking $k$ arguments and every $\{j_0,\ldots,j_{k-1}\} \subseteq \{0,\ldots,n-1\}$ one has $P(a_{j_0},\ldots,a_{j_{k-1}}) = P(b_{j_0},\ldots,b_{j_{k-1}})$.

  A metric structure $\cM$ is \textit{approximately ultrahomogeneous} if for any $n$-tuples $\bar a$ and $\bar b$ with the same quantifier-free type and any $\varepsilon >0$ there exists $g \in \Aut(\cM)$ such that $d(ga_i,b_i) \le \varepsilon$ for all $i < n$.
\end{dfn}

\begin{fct}
  \label{fct:StructurePolishGroup}
  Any Polish group $G$ is isomorphic (as a topological group) to the automorphism group of some approximately ultrahomogeneous Polish metric structure $\cM$ in a countable language, such that in addition $\partial_u = \partial_u^\cM$.
\end{fct}
\begin{proof}[Sketch of proof -- see \cite{Melleray:Oscillation}]
  Let $G$ be a Polish group, and $d$ be a left-invariant distance compatible with the topology of $G$.
  We denote by $\hat{G}$ the metric completion of $(G,d)$, and let $G$ act on $\hat{G}$ by left-translation, and on $\hat{G}^n$ by the diagonal product of this action.
  For any $n$, and any $G$-orbit ${\mathcal O}$ in ${\hat G}^n$, we let $R_{\mathcal O}  \colon \hat{G}^n \to [0,1]$ be defined by
  $$R_{\mathcal O}(\bar x) = d(\bar x, \cO).$$
  This is a uniformly continuous function on ${\hat G}$, and one can consider the relational Polish metric structure
  \begin{gather*}
    {\mathcal G}=(\hat{G},d,\{R_{\mathcal O}\}).
  \end{gather*}
  It is then standard to check that $G$ is the automorphism group of ${\mathcal G}$, that $\partial_u = \partial_u^\cG$, and that ${\mathcal G}$ is approximately ultrahomogeneous.

  Since the distance to any orbit can be obtained as a uniform limit of distances to orbits in any given dense family, and since the family of orbits in $\hat G^n$ is separable, we can have essentially the same structure in a countable language.
  In addition, we can always truncate the distance at one, thereby making it bounded (however, while the boundedness of predicates is important when dealing with actual continuous logic, compactness, ultra-products, and so on, when dealing solely with automorphism groups it is of little or no significance).
\end{proof}

A word of caution is in order here: there may be more than one way to present a given Polish group $(G,\tau)$ as the automorphism group of some Polish metric structure $\cM$.
Attached to each such structure $\cM$ comes a topometric structure $(G,\tau,\partial_u^\cM)$, and \emph{different structures may induce different metrics on $G$}.
\fref{fct:StructurePolishGroup} merely assures us that the canonical uniform distance \emph{can} be obtained in this manner.

Metric logical structures, as presented here, were defined in \cite{BenYaacov-Usvyatsov:CFO} as a basis for the semantics of continuous first order logic (the reader may also consult \cite{BenYaacov-Berenstein-Henson-Usvyatsov:NewtonMS} for a general survey regarding metric model theory, for whose purpose continuous logic was introduced).
Since we shall not use the logic of these structures in this paper, the basic vocabulary presented above should be enough to understand what follows.

\section{Ample generics for topometric groups}
\label{sec:TopometricAmpleGenerics}

We start by recalling:
\begin{dfn}
  \label{dfn:TopometricAmpleGenerics}
  Let $G$ be a topometric group acting on a topometric space $X$.
  The action has \emph{ample generics} if for any $n \in \bN$ and any $\varepsilon >0$ there exists some $\bar x \in X^n$ such that $(G {\cdot} \bar x)_\varepsilon$ is co-meagre in $X^n$.
\end{dfn}

Let us make a few observations about this definition.
First, since $G$ acts by $\partial$-isometries, we have for all $\bar x, \bar y \in X^n$ and $\varepsilon > 0$,
\begin{gather*}
  \bar x \in (G {\cdot} \bar y)_\varepsilon \quad \Longleftrightarrow \quad \bar y \in (G {\cdot} \bar x)_\varepsilon \quad \Longrightarrow \quad \left(G {\cdot} \bar x \right)_\varepsilon \subseteq \left(G {\cdot} \bar y \right)_{2\varepsilon},
  \intertext{whereby, denoting by $\overline{\,\cdot\,}^\partial$ the $\partial$-closure,}
  \bar y \in \overline{G {\cdot} \bar x}^{\partial} \quad \Longleftrightarrow \quad \overline{G {\cdot} \bar x}^{\partial} \cap \overline{G {\cdot} \bar y}^{\partial} \ne \emptyset \quad \Longleftrightarrow \quad \overline{G {\cdot} \bar x}^{\partial}=\overline{G {\cdot} \bar y}^{\partial}.
\end{gather*}
If the action of $G$ on $X$ has ample generics, the remarks above show that for all $\varepsilon >0$ there is a co-meagre set of $\bar x \in X^n$ such that
$\left(G {\cdot} \bar x \right)_\varepsilon$ is co-meagre; taking the intersection for all rational $\varepsilon >0$, we obtain that there is a co-meagre set $\mathcal O_n$ in $X^n$ such that $\overline{G {\cdot} \bar x}^{\partial}$ is co-meagre for all $\bar x \in \mathcal O_n$.
In particular, having ample generics is equivalent to having orbits in each $X^n$ whose $\partial$-closure is co-meagre.

Elements of $\mathcal O_n$ are called \emph{generics}.
If $\bar x, \bar y \in \mathcal O_n$ then $\overline{G {\cdot} \bar x}^{\partial}$ and $\overline{G {\cdot} \bar y}^{\partial}$ must intersect since they are both co-meagre, whereby $\overline{G {\cdot} \bar x}^{\partial}=\overline{G {\cdot} \bar y}^{\partial}$.
For similar reasons, if $\bar x \in \cO_n$ and $\bar y \in \overline{G {\cdot} \bar x}^{\partial}$ then $\overline{G {\cdot} \bar x}^{\partial} = \overline{G {\cdot} \bar y}^{\partial}$ and thus $\bar y \in \cO_n$.
We conclude that $\cO_n = \overline{G {\cdot} \bar x}^{\partial}$ for any $\bar x \in \mathcal O_n$.

We shall not really be interested in general topometric actions of Polish topometric groups; rather, we shall focus on the action of $G$ on itself by conjugation.

\begin{dfn}
  Let $(G,\tau,\partial)$ be a Polish topometric group.
  We say that $(G,\tau,\partial)$ has \emph{ample generics} if the topometric action of $G$ on itself by conjugation has ample generics.
  If $\partial=\partial_u$, we simply say that $(G,\tau)$ has \emph{ample metric generics}.
\end{dfn}
Notice that since $\partial_u$ is always coarser than (or equivalent to) $\partial$, if $(G,\tau,\partial)$ has ample generics then so does $(G,\tau,\partial_u)$.
In other words, a Polish group $(G,\tau)$ has ample metric generics if and only if there exists $\partial$ such that $(G,\tau,\partial)$ is a topometric group with ample generics.

Our next goal is to extend the techniques of Kechris and Rosendal to the topometric setting.
We first need to establish a few technical lemmas, extending \cite[Lemmas 6.6 and 6.7]{Kechris-Rosendal:Turbulence}.
The proofs are very similar, but require one to keep account of some $\varepsilon$'s.

\begin{conv}
  For the rest of this section, $G$ is a Polish topometric group acting on a Polish topometric space $X$, and we assume that the action has ample generics.
  As above, we let $\cO_n \subseteq X^n$ denote the set of generic tuples, recalling that this is a co-meagre set, equal to $\overline{G {\cdot} \bar x}^\partial$ for any $\bar x \in \cO_n$.
\end{conv}

\begin{lem}
  \label{lem:0}
  Under our assumptions,
  \begin{gather*}
    \cO_n = \bigl\{ \bar x \in X^n \colon \forall^* y \in X \  (\bar x,y) \in \cO_{n+1} \bigr\}.
  \end{gather*}
\end{lem}
\begin{proof}
  Let the set on the right hand side be denoted by $D$.
  Clearly $\cO_n \supseteq D$, and $D$ is $G$-invariant.
  By the Kuratowski-Ulam Theorem $D$ is co-meagre, and in particular non empty, so all we need to show is that $D$ is $\partial$-closed.
  Indeed, let $(\bar x_k)$ be a sequence in $D$ $\partial$-converging to $\bar x \in X^n$.
  Then
  $$\forall k \  \forall^* y \in X \  (\bar x_k,y) \in \cO_{n+1}.$$
  Since the intersection of countably many co-meagre sets is co-meagre,
  $$\forall^* y \in X \   \forall k \  (\bar x_k,y) \in \cO_{n+1}.$$
  Since $\cO_{n+1}$ is $\partial$-closed, this implies that $\bar x \in D$, and the proof is complete.
\end{proof}

\begin{lem}
  \label{lem:1}
  Let $A,B \subseteq X$ be such that $A$ is not meagre and $B$ is not meagre in any non empty open set.
  Then if $\bar x \in X^n$ is generic and $V$ is an open neighbourhood of the identity of $G$, there exist for any $\varepsilon>0$ some $y_0 \in A$, $y_1 \in B$, $h \in V$ such that $(\bar x,y_0)$ and $(\bar x,y_1)$ are generic and $\partial\bigl( h{\cdot}(\bar x,y_0), (\bar x,y_1) \bigr) < \varepsilon$.
\end{lem}
\begin{proof}
  Fix an open neighbourhood $V$ of $1_G$, $n \in \bN$, a generic $\bar x \in X^n$ and $\varepsilon > 0$.
  Let
  \begin{gather*}
    G_{\bar x,\varepsilon}= \{g \in G \colon \partial(g{\cdot}\bar x,\bar x) < \varepsilon\}, \qquad E = \{y \in X \colon (\bar x,y) \in \cO_{n+1}\}.
  \end{gather*}
  By \fref{lem:0}, $E$ is co-meagre, so we may choose $y_0 \in E \cap A$.
  Notice that for any $y \in E$ we have $(\bar x,y) \in \cO_{n+1} = \overline{G {\cdot} (\bar x,y_0)}^\partial$, so there exists $g \in G_{\bar x,\varepsilon}$ such that $\partial(y, g {\cdot} y_0) < \varepsilon$.
  Thus we have
  \begin{gather*}
    E \subseteq \left(G_{\bar x, \varepsilon} {\cdot} y_0\right)_\varepsilon = \bigcup_{g \in G_{\bar x, \varepsilon}} \bigl( ( gV \cap G_{\bar x, \varepsilon} ) {\cdot} y_0 \bigr)_\varepsilon.
  \end{gather*}
  Since $V$ is open and $G_{\bar x, \varepsilon}$ has a countable basis, there must exist a sequence $(g_n) \subseteq G_{\bar x, \varepsilon}$ such that
  \begin{gather*}
    E \subseteq \bigcup_n \bigl( (g_nV \cap G_{\bar x, \varepsilon}){\cdot} y_0 \bigr)_\varepsilon.
  \end{gather*}
  Since $E$ is co-meagre, there exists some $n$ such that $\bigl( (g_nV \cap G_{\bar x,\varepsilon}) {\cdot} y_0 \bigr)_\varepsilon$ is non-meagre, and hence $\bigl( (V \cap G_{\bar x,2\varepsilon}) {\cdot} y_0 \bigr)_\varepsilon$ is also non-meagre.
  Since this set is analytic it must be Baire-measurable and so it is co-meagre in some non empty open subset of $X$, so it must intersect $B \cap E$.
  Let $y_1 \in  \bigl( (V \cap G_{\bar x,2\varepsilon}) {\cdot} y_0 \bigr)_\varepsilon \cap B \cap E$ and let $h \in V \cap G_{\bar x,2\varepsilon}$ be such that $\partial(h{\cdot} y_0,y_1)< \varepsilon$.

  Then $(\bar x,y_0)$ and $(\bar x,y_1)$ are generic, $h \in V$, $\partial(h{\cdot}\bar x,\bar x) < 2 \varepsilon$ and $\partial(h{\cdot} y_0,y_1)< \varepsilon$, which is enough.
\end{proof}

\begin{lem}
  \label{lem:2}
  Let $(A_n), (B_n)$ be two sequences of subsets of $X$ such that for any $n$ $A_n$ is not meagre and $B_n$ is not meagre in any non empty open set.
  Let also $(r_n)$ be a sequence of strictly positive reals.
  Then there exists a continuous map $a \mapsto h_a$ from $2^\bN$ into $(G,\tau)$ such that if $a\rest_n = b\rest_n$ and $a(n)=0$, $b(n)=1$ then $\partial(h_a {\cdot} A_n,h_b {\cdot} B_n)<r_n$.
\end{lem}
\begin{proof}
  First, fix a sequence of strictly positive reals $(\varepsilon_n)$ such that for all $n$
  $$2 \sum_{m \geq n} \varepsilon_m < r_n.$$
  Let $d$ be a complete distance in $G$ inducing $\tau$.
  We denote by $2^{< \omega}$ the set of finite binary sequences; for $s \in 2^{<\omega}$ we denote its length by $|s|$.

  We now define for any $s \in 2^{< \omega}$ elements $h_s \in G$ and $x_s \in X$ in such a way that
  \begin{enumerate}
  \item $\bar x_s = (x_{s\rest_1},\ldots, x_s)$ is generic.
  \item $x_{s \concat 0} \in A_{|s|}$ and $x_{s \concat 1} \in B_{|s|}$.
  \item \label{item:lem2.3}
    $h_{s \concat 0} = h_s$.
  \item $d(h_s,h_{s \concat 1})< 2^{-|s|}$.
  \item \label{item:lem2.5}
    $\partial(h_{s \concat 0} {\cdot} \bar x_{s \concat 0}, h_{s \concat 1} {\cdot} \bar x_{s \concat 1}) < \varepsilon_{|s|}$.
  \end{enumerate}
  We proceed by induction on $|s|$, starting with $h_\emptyset = \id$ and $x_\emptyset$ arbitrary since it is never used.
  Now assume that $h_s$ and $\bar x_s$ are given.
  Applying \fref{lem:1}, we can find $x_{s\concat 0}$, $x_{s \concat 1}$ and $f_s$ such that $\bar x_{s \concat 0}$, $\bar x_{s \concat 1}$ are generic, $d(h_s,h_sf_s) < 2^{-|s|}$, $x_{s\concat 0} \in A_{|s|}$, $x_{s\concat 1} \in B_{|s|}$, and finally $\partial(f_s {\cdot} \bar x_{s \concat 1}, \bar x_{s \concat 0}) < \varepsilon_{|s|}$.
  Then $h_{s\concat 0} = h_s$ and $h_{s\concat 1} = h_sf_s$ will do.

  Once the construction is done, for every $a \in 2^\bN$ we can let $h_a= \lim h_{a\rest_n}$, so the map $a \mapsto h_a$ is continuous.
  It follows from the properties above that $\partial(h_s {\cdot} \bar x_s, h_{s \concat i} {\cdot} \bar x_s) < \varepsilon_{|s|}$ for either value of $i \in \{0,1\}$, and the same hold \textit{a fortiori} if we replace $x_s$ with any initial segment thereof.
  Therefore, for all $m < n$,
  \begin{gather*}
    \partial(h_{a\rest_n} {\cdot} \bar x_{a\rest_m},h_{a\rest_m} {\cdot} \bar x_{a\rest_m}) \leq \sum_{m \leq k < n} \varepsilon_k < \sum_{m \leq k} \varepsilon_k.
  \end{gather*}
  Fixing $m$ and letting $n \to \infty$, and since $\partial$ is lower semi-continuous, we obtain
  \begin{gather*}
    \partial(h_a {\cdot} \bar x_{a\rest_m},h_{a\rest_m} {\cdot} \bar x_{a\rest_m}) \le \sum_{k \geq m} \varepsilon_k.
  \end{gather*}
  Now let $a,b \in 2^\bN$ and assume that $a\rest_n = b\rest_n = s$, $a(n)=0$, $b(n)=1$.
  Then
  \begin{align*}
    \partial(h_a {\cdot} A_n,h_b {\cdot} B_n) & \leq \partial(h_a {\cdot} \bar x_{s\concat 0},h_b {\cdot} \bar x_{s\concat 1}) \leq \partial(h_{s\concat 0} {\cdot} \bar x_{s\concat 0},h_{s\concat 1} {\cdot} \bar x_{s\concat 1}) + 2 \sum_{k> n} \varepsilon_k \leq \varepsilon_{n}+ 2 \sum_{k> n} \varepsilon_k < r_n,
  \end{align*}
  completing the proof.
\end{proof}

\section{Consequences of ample generics}
\label{sec:Consequences}

\subsection{Automatic continuity for topometric groups}

Recall the following definitions from \cite{Rosendal-Solecki:AutomaticContinuity}.

\begin{dfn}
  \label{dfn:Steinhaus}
  If $G$ is a group, we say that $A\subseteq G$ is \emph{$\sigma$-syndetic} if $G$ is covered by countably many left-translates of $A$.
  A topological group $G$ has the \emph{Steinhaus property} if there exists some integer $k$ such that, whenever $A$ is symmetric, $\sigma$-syndetic, we have $1_G\in \Int(A^k)$.
\end{dfn}

This notion was introduced in \cite{Rosendal-Solecki:AutomaticContinuity} to study automatic continuity properties of homomorphisms.
It is easy to see that if $G$ is a topological group with the Steinhaus property, then any homomorphism from $G$ to a second countable topological group $H$ is continuous (see \fref{prp:Steinhaus} below).
We introduce a topometric analogue of this property.

\begin{dfn}
  A topometric group has the \emph{Steinhaus property} if there exists some integer $k$ such that, whenever $A$ is symmetric, $\sigma$-syndetic, we have $1\in \Int \bigl( (A^k)_\varepsilon \bigr)$ for all $\varepsilon >0$ (the interior being with respect to the topology).
  We shall then say that $G$ is Steinhaus with exponent $k$.
\end{dfn}

\begin{prp}
  \label{prp:Steinhaus}
  Let $(G,\tau,\partial)$ be a Polish topometric group with the Steinhaus property, and $H$ be a secound countable topological group.
  Assume that $\varphi \colon (G,\partial) \to H$ is a continuous morphism.
  Then $\varphi \colon (G,\tau) \to H$ is continuous.
\end{prp}
\begin{proof}
  Let $G$ be Steinhaus with exponent $k$.
  Let $V$ be a neighbourhood of $1_H$ and pick some symmetric neighbourhood $W$ of $1_H$ such that $W^{k+1} \subseteq V$.
  Then there exists some $\varepsilon >0$ such that $(1_G)_\varepsilon \subseteq \varphi^{-1}(W)$.
  Also, since $\varphi^{-1}(W)$ is $\sigma$-syndetic, we know that
  \begin{gather*}
    1_G \in \Int\bigl( (\varphi^{-1}(W)^k)_\varepsilon \bigr) \subseteq (1_G)_\varepsilon {\cdot} \varphi^{-1}(W)^k \subseteq \varphi^{-1}(W^{k+1}) \subseteq \varphi^{-1}(V).
  \end{gather*}
  Thus $\varphi^{-1}(V)$ is a $\tau$-neighbourhood of $1_G$, which suffices for showing that $\varphi \colon (G,\tau) \to H$ is continuous.
\end{proof}

\begin{cor}
  If $(G,\tau,\partial)$ is a Polish topometric group with the Steinhaus property then for any Polish topology $\tau'$ such that $(G,\tau',\partial)$ is also a topometric group we must have $\tau=\tau'$.
\end{cor}

\begin{thm}
  Let $G$ be a topometric Polish group with ample generics.
  Then $G$ is Steinhaus with exponent $10$.
\end{thm}

This is a corollary of the following result (again an analogue of a theorem due to Kechris and Rosendal), which we shall also use in \fref{sec:Bergman}.

\begin{thm}
  \label{thm:Cosets}
  Let $G$ be a Polish topometric group with ample generics.
  Let $k_n \in G$ and $A_n \subseteq G$ for $n \in \bN$, such that $\bigcup_{n \in \bN} k_n A_n$ is co-meagre in $G$.
  Then there is some $n$ such that for all $\varepsilon > 0$,
  \begin{gather*}
    1_g \in \Int\bigl( (A_n^{-1}A_nA_n^{-1}A_n^{-1}A_nA_n^{-1}A_nA_nA_n^{-1}A_n)_\varepsilon \bigr).
  \end{gather*}
\end{thm}
\begin{proof}
  If there is some $n$ such that $(A_n^{-1}A_nA_nA_n^{-1}A_n)_\varepsilon$ is co-meagre in some non empty open set for all $\varepsilon >0$, then Pettis' theorem yields the desired result.
  We may therefore assume there exists for each $n$ some $r_n>0$ such that $(A_n^{-1}A_nA_nA_n^{-1}A_n)_{r_n}$ is not co-meagre in any non empty open set.
  We may also assume that each $A_n$ is non meagre, and repeated in the sequence infinitely often.

  Let $B_n = G \setminus (A_n^{-1}A_nA_nA_n^{-1}A_n)_{r_n}$, and apply \fref{lem:2} to $(A_n)$, $(B_n)$, $(r_n)$ to obtain a continuous map $a \mapsto h_a$ from $2^\bN$ into $G$ as stated there.
  If $a,b \in 2^\bN$ are distinct, say $a\rest_n = b\rest_n$ and $a(n) = 0$, $b(n) = 1$, then  $\partial(h_aA_nh_a^{-1},h_bB_nh_b^{-1})<r_n$.
  Since $\partial(A_n,B_n) \ge r_n$ and $\partial$ is bi-invariant, we must have $h_a \ne h_b$.
  Therefore $a \mapsto h_a$ is injective, and being continuous from a compact space to a Hausdorff one, it is a homeomorphic embedding, with image $H = \{h_a \colon a \in 2^{\bN}\}$.
  Let also $A = \bigcup_n k_n A_n$, which is, by hypothesis, co-meagre.
  Since the map $(g,h) \mapsto g^{-1}h$ is open and continuous from $G \times H$ into $G$, and the inverse image of a co-meagre set by such a map is co-meagre, Kuratowski-Ulam yields
  \begin{gather*}
    \forall^* g \in G \ \forall^* h \in H \ g^{-1}h \in A.
  \end{gather*}
  In particular, there exists $g \in G$ such that
  $$\forall^{*} h \in H \  h \in gA.$$
  Then there exists $n$ such that $\{a \in 2^\bN \colon h_a \in g k_n A_n\}$ is non-meagre.
  In particular, its closure cannot have empty interior and so this set must be dense in $V_t=\{a \in 2^\bN \colon t \subseteq a\}$ for some $t \in 2^{<\omega}$.
  Choose $m > |t|$ such that $A_n=A_m$, and let $s$ be any extension of $t$ of length $m$.
  Then there exist $a$ and $b$ such that $a\rest_m = b\rest_m = s$, $a(m)=0$, $b(m)=1$ and $h_a,h_b \in g k_n A_n = g k_n A_m$.
  Set $h_a = g k_n h_1$ and $h_b=g k_n h_2$.
  Then
  $$h_b^{-1}h_aA_mh_a^{-1}h_b=h_2^{-1}h_1A_mh_1^{-1}h_2 \subseteq A_m^{-1}A_mA_mA_m^{-1}A_m.$$
  Thus, $\partial(h_b^{-1}h_aA_mh_a^{-1}h_b,B_m) >r_n$, contradicting the choice of $h_a,h_b$ and completing the proof.
\end{proof}

As far as automatic continuity theorems go, one can actually do better.
Remember that a group $G$ has \emph{uniform Souslin number} $\kappa$ if $\kappa$ is the least cardinal such that, for any neighbourhood $V$ of $1_G$, $G$ can be covered by less than $\kappa$ many left translates of $V$.
Equivalently, $\kappa$ is the least cardinal such that for any open neighbourhood $V$ of $1_H$ there do not exist $\kappa$ disjoint left translates of $V$.
The techniques of \cite{Kechris-Rosendal:Turbulence} yield the following theorem.

\begin{thm}
  \label{thm:AutomaticContinuity}
  Let $(G,\tau,\partial)$ be a Polish topometric group with ample generics, $H$ a topological group with uniform Souslin number $<2^{\aleph_0}$, and $\varphi \colon G \to H$ a morphism that is continuous from $(G,\partial)$ into $H$.
  Then $\varphi$ is continuous from $(G,\tau)$ to $H$.
\end{thm}
\begin{proof}
  Let $G,H,\varphi$ be as in the statement of the theorem, and let $W$ be a neighbourhood of $1_H$.
  Choose some symmetric open neighbourhood $V$ of $1_H$ such that $V^{21} \subseteq W$, as well as $\varepsilon >0$ such that $(1_G)_{2\varepsilon} \subseteq \varphi^{-1}(V)$.
  Finally, let $A=\varphi^{-1}(V^2)$.

  We first claim that $A$ is non meagre.
  Indeed, if $A$ were meagre then there would be, by the Kuratowski-Mycielski Theorem (see \cite[19.1]{Kechris:Classical}), a Cantor set $C \subseteq G$ such that $g^{-1}h \notin A$ for any $g \neq h \in C$.
  Since
  \begin{gather*}
    \varphi(g) V\cap \varphi(h)V \neq \emptyset \quad \Longleftrightarrow \quad \varphi(g^{-1}h) \in V^2 \quad \Longleftrightarrow \quad g^{-1}h \in A,
  \end{gather*}
  there would be uncountably many disjoint left translates of $V$ in $H$, contradicting the assumption on its uniform Suslin number.

  If $(A^5)_\varepsilon$ is co-meagre in some non empty open set, then by Pettis' theorem $(A^{10})_{2\varepsilon}$ contains a neighbourhood $O$ of $1_G$.
  By choice of $V$ and $\varepsilon$, one has
  $$\varphi(O) \subseteq \varphi(A^{10})V \subseteq V^{21} \subseteq W.$$

  Towards a contradiction, let us therefore assume that $G \setminus (A^5)_\varepsilon$ is non meagre in every non empty open set.
  Applying \fref{lem:2} to $A$, $G \setminus (A^5)_\varepsilon$ and $\varepsilon$, we can find $h_{a} \in G$ for $a \in 2^\bN$ such that if $a\rest_n = b\rest_n$, $a(n)=0$, $b(n)=1$ then
  $$\partial\bigl( h_b^{-1}h_aAh_a^{-1}h_b,G \setminus (A^5)_\varepsilon \bigr) < \varepsilon.$$
  Since $A$ covers $G$ by fewer than $2^{\aleph_0}$ left translates, there are distinct $a,b \in 2^\bN$ and some $g \in G$ such that $h_a,h_b \in gA$.
  Say $a\rest_n = b\rest_n$, $a(n)=0$ and $b(n)=1$, and let $g_a,g_b \in A$ be such that $h_a=gg_a$, $h_b=gg_b$.
  Then
  $$h_b^{-1}h_aAh_{a}^{-1}h_b = g_b^{-1}g_aAg_{a}^{-1}g_b  \subseteq A^5.$$
  Hence $\partial(A^5,G \setminus (A^5)_\varepsilon) < \varepsilon$, a contradiction, which concludes the proof.
\end{proof}

\subsection{A weak small index property for topometric groups}

Recall that a Polish group $G$ has the \emph{small index property} if any subgroup $H \le G$ of index strictly less than $2^{\aleph_0}$ is open.
Kechris and Rosendal proved that a Polish group with ample generics must have the small index property.
In our context, their techniques yield the following result.

\begin{prp}
  \label{prp:SmallIndex}
  Assume $(G,\tau,\partial)$ is a Polish topometric group with ample generics.
  Then any $\partial$-closed subgroup of $G$ of index strictly less than $2^{\aleph_0}$ is open.
\end{prp}

Said differently, the $\partial$-closure of a subgroup of index strictly less than $2^{\aleph_0}$ is open (hence clopen).
This result is not very useful, for most groups we have in mind are connected, and thus do not have any open subgroups.
So, for these groups, \fref{prp:SmallIndex} merely says that they have no $\partial$-closed subgroups of index $< 2^{\aleph_0}$.
Since the proof is a straightforward adaptation of Kechris and Rosendal's proof, we don't include it and content ourselves with stating the above Proposition for the record.

\subsection{Bergman property for groups with a bi-invariant metric}
\label{sec:Bergman}

\begin{dfn}
  An abstract group $G$ has the \emph{Bergman property} if whenever $W_0 \subseteq W_1 \subseteq \ldots \subseteq W_n \ldots$ is an increasing, exhaustive sequence of subsets of $G$ there exist $n$ and $k$ such that $G=W_n^k$.
  If $k$ above can be chosen independently of the sequence $(W_n)$, then we say that $G$ is \emph{$k$-Bergman}.
\end{dfn}

In \cite{Bergman:GeneratingSubsets}, G.\ Bergman proved that ${\mathfrak S}_{\infty}$ has the Bergman property.
It was noticed independently by Y.\ de Cornulier and V.\ Pestov that the Bergman property is equivalent to each of the following statements:
\begin{enumerate}
\item Any left-invariant pseudo-metric on $G$ is bounded
\item Any action by isometries of $G$ on a metric space $X$ has bounded orbits.
\end{enumerate}

In \cite{Rosendal:PropertyOB}, C.\ Rosendal introduced a variant of the Bergman property for topological groups, which he dubbed \emph{property (OB)}; a topological group has this property if it satisfies the following condition:
\begin{quote}
  Whenever $G$ acts by isometries on a metric space $X$ in such a way that for all $x$ the mapping $g \mapsto g {\cdot} x$ is continuous, the action of $G$ has bounded orbits.
\end{quote}

Let us now introduce a new, and closely related, property.

\begin{dfn}
  Let $G$ be a group endowed with a bi-invariant (and, as usual, not necessarily separable) distance $\partial$.

  We say that $(G,\partial)$ has the \emph{metric Bergman property} if whenever $(W_n)$ is an increasing, exhaustive sequence of subsets of $G$ then for any $\varepsilon >0$ there exists $n$ and $k$ such that $(W_n)_\varepsilon^k = G$ (where $(W_n)_\varepsilon^k$ can only be parsed as $\bigl( (W_n)_\varepsilon \bigr)^k$).
  If $k$ above can be chosen independently from the sequence $(W_n)$ and $\varepsilon$, then we say that $(G,\partial)$ is \emph{metrically $k$-Bergman}.
\end{dfn}

Notice that $(W^k)_\varepsilon \subseteq (W)_\varepsilon^k \subseteq (W^k)_{k\varepsilon}$, so $(G,\partial)$ is metrically $k$-Bergman if and only if, under the same hypotheses, there exists $n$ such that $(W_n^k)_\varepsilon = G$.
If $\partial$ is discrete then the metric ($k$-)Bergman property is equivalent to the ordinary ($k$-)Bergman property, and can therefore be seen as a generalisation thereof.
The equivalences above become:
\begin{prp}
  Let $G$ be a  group endowed with a bi-invariant distance $\partial$.
  Then the following are equivalent:
  \begin{enumerate}
  \item $(G,\partial)$ has the metric Bergman property.
  \item Whenever $d$ is a left-invariant pseudo-metric on $G$ such that for some $\varepsilon > 0$, the set $\{g \colon \partial(g,1) < \varepsilon\}$ is bounded in $d$, $d$ must be bounded on $G$.
  \item Whenever $(G,\partial)$ acts by isometries on a metric space $(X,d)$ in such a way that for all $x$ there exists $\varepsilon >0$ with $\partial(g,1) < \varepsilon \ \Longrightarrow \  d(x,gx) \le 1/\varepsilon$, all $G$-orbits are bounded in $(X,d)$.
  \end{enumerate}
\end{prp}
\begin{proof}
  It is easy to check that the last two conditions are equivalent, so we only show that (i) and (ii) are equivalent.
  \begin{cycprf}
  \item[\impnext]
    Let $d$ be a left-invariant pseudo-metric on $G$, and $M \in \bN$ such that $\partial(g,1) < 1/M \ \Longrightarrow \ d(g,1) \leq M.$
    Let $W_n = \{g\colon d(g,1) \leq n\}$; by the metric Bergman property for $(G,\partial)$ there are $n$ and $k$ such that $G = (W_n)_{1/M}^k$, where $(\cdot)_\varepsilon$ is with respect to $\partial$.
    Now, $(W_n)_{1/M} = W_n {\cdot} (1)_{1/M} \subseteq W_{n+M}$, so $G \subseteq W_{n+M}^k$, whereby $d(g,1) \leq (n+M)k$ for all $g \in G$.
  \item[\impfirst]
    Let $(W_n)_n$ be an increasing, exhaustive sequence of subsets of $G$, and $\varepsilon > 0$.
    Then $(W_n \cap W_n^{-1})_n$ is still exhaustive, and so one can assume without loss of generality that each $W_n$ is symmetric and $W_0=\{1\}$.
    One can define a left-invariant pseudo-metric $d$ on $G$ by setting
    \begin{gather*}
      d(f,g)= \inf\{k_1+\ldots+k_n \colon g^{-1}f = h_1\ldots h_n, \ h_i \in (W_{k_i})_\varepsilon \}.
    \end{gather*}
    By definition, $d(f,1) \leq 1$ for all $f \in (1)_\varepsilon$ and so $d$ is bounded.
    This is only possible if there is some $k$ and some $n$ such that $G=(W_k)_\varepsilon^n$, which shows that $(G,\partial)$ has the metric Bergman property.
  \end{cycprf}
\end{proof}

Note that the Bergman property for $G$ implies the metric Bergman property for $(G,\partial)$, which in turn implies property (OB) for $(G,\partial)$.

\begin{dfn}
  \label{dfn:ApproximatelyOligomorphicAction}
  Let $G$ act by isometries on some Polish metric space $M$.
  We say that the action is \emph{approximately oligomorphic} if for any $n \in \bN$ and any $\varepsilon >0$ there exist $\bar b_1,\ldots,\bar b_k \in M^n$ such that for each $\bar x \in M^n$ there  exists $j \in \{1,\ldots,k \}$ and $g \in G$ satisfying $d(g{\cdot} \bar b_j,\bar x) \le \varepsilon$.
\end{dfn}

The action of the isometry group of the Urysohn space $\bU_1$ of diameter $1$ on $\bU_1$ is an example of approximately oligomorphic action\footnote{See the next section for a definition and a brief discussion of Urysohn spaces.}.

We then have the following analogue of \cite[Theorem~6.19]{Kechris-Rosendal:Turbulence}.

\begin{thm}
  \label{thm:AmpleGenericsMetricBergman}
  Suppose $(X,d)$ is a Polish metric space and $G$ is a closed subgroup of $\Iso(X,d)$ whose action on $(X,d)$ is approximately oligomorphic.
  Let $\tau$ denote the topology of point-wise convergence, $\partial$ denote the metric of uniform convergence, and assume that $(G,\tau,\partial)$ has ample generics.
  Then $(G,\partial)$ is metrically $21$-Bergman.
\end{thm}
\begin{proof}
  Let $W_n$ be an increasing, exhaustive sequence of symmetric subsets of $G$, and let $\varepsilon >0$.
  By \fref{thm:Cosets}, there is some $n$ such that $(W_n^{10})_\varepsilon$ is a neighbourhood of $1$.
  Rosendal \cite[Theorem~5.2]{Rosendal:PropertyOB} proves that $G$ is Roelcke-precompact (this is also an easy consequence of the Ryll-Nardzewski theorem in continuous logic), i.e., for any neighbourhood $V$ of $1$ there is a finite subset $F$ of $G$ such that $G=VFV$.
  We may apply this to $(W_n^{10})_\varepsilon $, find $m \ge n$ such that $F \subseteq W_m$, and then
  $$G \subseteq (W_m^{10})_\varepsilon W_m (W_m^{10})_\varepsilon \subseteq (W_m^{21})_{2\varepsilon}.$$
  Since $\varepsilon$ was arbitrary, this concludes the proof.
\end{proof}

We note that the uniformity generated by $\partial$ may not be the coarsest bi-invariant uniformity refining $\tau$ (it may be strictly finer).

\section{A criterion for ample generics}
\label{sec:Criterion}

Having shown consequences of topometric ample generics, it is time to provide a criterion for their existence.
We prove a \emph{sufficient} condition, which says roughly that if $\cM$ is a Polish metric structure, with some countable dense ``nice substructure'' $\cN$, and in addition, viewing $\cN$ as a discrete structure, $\Aut(\cN)$ has ample generics, then $\Aut(\cM)$ has ample generics as a topometric group\footnote{In subsequent work \cite{BenYaacov:AmpleGenerics} the first author proves a necessary and sufficient condition for a topometric automorphism group to have ample generics, generalising the criterion appearing in \cite{Kechris-Rosendal:Turbulence}.}.

\def\version{2}
\if\version 1
\begin{thm}
  \label{thm:ApproximationCriterionAbstract}
  Let $(X,\tau_X,\partial_X)$ be a Polish topometric space, $(Y,\tau_Y)$ a Polish topological space, and $f\colon (Y,\tau_Y) \to (X,\tau_X)$ a continuous map with $\tau_X$-dense image such that for every open $U \subseteq Y$ and $\varepsilon > 0$, the set $(fU)_\varepsilon \subseteq X$ (in the sense of $\partial_X$) is open as well.
  Assume also that for every open $V \subseteq X$ and $\varepsilon > 0$, the set $(V)_\varepsilon \subseteq X$ is open.

  Then, if $A \subseteq Y$ is co-meagre then so is $\overline{fA}^{\partial_X} \subseteq X$.
\end{thm}
\begin{proof}
  It will be enough to show that if $A \subseteq Y$ is co-meagre and $\varepsilon > 0$ then $(fA)_{2\varepsilon} \subseteq X$ is co-meagre.
  We shall do it in a manner which makes implicit use of the Banach-Mazur game criterion for co-meagreness (see \cite[Chapter I, Section 8]{Kechris:Classical}).

  Let $A \supseteq \bigcap O_n$ where each $O_n \subseteq Y$ is a dense open set, noticing that then $(f O_n)_\varepsilon$ is a dense open set in $X$.
  Let us also fix a complete metrisation $d_Y$ of $\tau_Y$.
  Using Zorn's Lemma, construct families $S_n$ of pairs $(V,U)$, where $V \subseteq Y$ and $U \subseteq X$ are open, maximal subject to the following constraints:
  \begin{enumerate}
  \item $S_0 = \{(Y,X)\}$.
  \item If $(V,U) \in S_{n+1}$ then $U \subseteq (fV)_\varepsilon$, $\diam_{d_Y}(V) < 2^{-n}$, and there is $(V',U') \in S_n$ such that $\overline{V} \subseteq O_n \cap V'$ and $U \subseteq U'$.
  \item If $(V,U),(V',U') \in S_n$ are distinct then $U \cap U' = \emptyset$.
  \end{enumerate}
  We claim that $W_n = \bigcup_{(V,U) \in S_n} U$ is dense for all $n$.
  Indeed, if not, then there is $W \subseteq X$ open, non empty, and disjoint from $W_n$ for some minimal $n > 0$.
  Then $W \cap U_{n-1} \neq \emptyset$, and we may assume that $W \subseteq U'$ for some $(V',U') \in S_{n-1}$.
  Then $f^{-1}[(W)_\varepsilon] \cap V'$ is open and non empty in $Y$, so we can find $V \subseteq Y$ open such that $\overline V \subseteq f^{-1}[(W)_\varepsilon] \cap V' \cap O_n$ of the desired diameter.
  Going back, $U = (fV)_\varepsilon \cap W$ is open and non empty, and we may add $(V,U)$ to $S_n$, contradicting maximality, and proving our claim.

  We have shown that $\bigcap W_n$ is co-meagre in $X$, and we next claim that $(fA)_{2\varepsilon} \supseteq \bigcap W_n$.
  Indeed, let $x \in \bigcap W_n$.
  By the disjointness assumption, there exists a unique sequence of $(V_n,U_n) \in S_n$ such that $x \in \bigcap U_n$.
  For each $n$ we have $y_n \in V_n$ such that $\partial(fy_n,x) < \varepsilon$.
  Moreover, the constraints we imposed yield that $A \supseteq \bigcap O_n \supseteq \bigcap V_n = \{y\}$ where $y = \lim y_n$.
  Since $\partial$ is topologically lower semi-continuous we have $\partial(x,fy) \leq \varepsilon$, so indeed $x \in (fA)_{2\varepsilon}$.

  We conclude that $(fA)_{2\varepsilon}$ is co-meagre, as desired.
\end{proof}

\else

Let us first recall the characterisation of co-meagre sets via Banach-Mazur games (see \cite[Chapter I, Section 8]{Kechris:Classical}).
If $X$ is a topological space and $A$ is a subset of $X$, the Banach-Mazur game $G(A)$ is the game where two players I, II take turns playing non empty open sets $U_i$, $V_i$ in such a way that $U_{i+1} \subseteq V_i \subseteq U_i$ for all $i \in \bN$.
Player II wins a run of this game if $\bigcap U_i \subseteq A$, else Player I wins.
Schematically, a run of $G(A)$ may be represented as follows:
\begin{align*}
  \text{I} \qquad & U_0 \qquad \qquad U_1 \subseteq V_0  \qquad \qquad  \ldots  \qquad \ \qquad U_{n+1} \subseteq V_n \quad \hfill \ldots \\
  \text{II} \qquad & \qquad  V_0 \subseteq U_0 \qquad \ \quad \ldots  \qquad \qquad V_n \subseteq U_n \qquad \qquad \qquad \hfill \ldots \\
\end{align*}
A \emph{strategy} $\sigma$ for one of the players (say, Player II), is what one would expect: $\sigma$ gives Player II a unique way to respond to any move by Player I.
A strategy is \emph{winning} if following it ensures victory.

\begin{fct}
  \label{fct:BanachMazur}
  If $X$ is a Polish space and $A \subseteq X$ then Player II has a winning strategy in $G(A)$ if and only if $A$ is co-meagre in $X$.
\end{fct}

\begin{thm}
  \label{thm:ApproximationCriterionAbstract}
  Let $(X,\tau_X,\partial_X)$ be a Polish topometric space, $(Y,\tau_Y)$ a Polish topological space, and $f\colon (Y,\tau_Y) \to (X,\tau_X)$ a continuous map with $\tau_X$-dense image such that for every open $U \subseteq Y$ and $\varepsilon > 0$, the set $(fU)_\varepsilon \subseteq X$ (in the sense of $\partial_X$) is open as well.
  Assume also that for every open $V \subseteq X$ and $\varepsilon > 0$, the set $(V)_\varepsilon \subseteq X$ is open.

  Then, if $A \subseteq Y$ is co-meagre then so is $\overline{fA}^{\partial_X} \subseteq X$.
\end{thm}
\begin{proof}
  Let $A \subseteq Y$ be co-meagre.
  Let $A \supseteq \bigcap O_n$ where each $O_n \subseteq Y$ is a dense open set, noticing that then $(f O_n)_\varepsilon$ is a dense open set in $X$.
  Let us also fix a complete metrisation $d_Y$ of $\tau_Y$.

  It will be enough to show that for every $\varepsilon > 0$, $(fA)_{2\varepsilon}$ is co-meagre in $X$.
  For this it will be enough to construct a winning strategy for Player II of $G\bigl( (fA)_{2\varepsilon} \bigr)$, which we now do.

  Our strategy will construct a sequence of open sets $W_n \subseteq Y$ and always play the open set $V_n = U_n \cap (fW_n)_\varepsilon$, making sure it is also non empty.
  We start with $W_0 = O_0$, noticing that $V_0 = U_0 \cap (fO_0)_\varepsilon$ is non empty.
  For $n > 0$ we have $U_n \subseteq (fW_{n-1})_\varepsilon$.
  Then $W_{n-1} \cap f^{-1}[(U_n)_\varepsilon]$ is open and non empty, and thus we may choose $W_n \subseteq Y$ open non empty such that $\overline{W_n} \subseteq W_{n-1} \cap f^{-1}[(U_n)_\varepsilon] \cap O_n$ and $\diam_{d_Y}(W_n) < 2^{-n}$.
  Again, $V_n = U_n \cap (fW_n)_\varepsilon$, is non empty.

  Now assume that $x \in \bigcap U_n$ for some run of the game.
  Then for each $n$ there is $y_n \in W_n$ such that $\partial(fy_n,x) < \varepsilon$.
  By the way we chose the $W_n$, we have $A \supseteq \bigcap W_n = \{y\}$ where $y = \lim y_n$.
  Since $\partial$ is topologically lower semi-continuous we have $\partial(x,fy) \leq \varepsilon$.
  Thus $x \in (fA)_{2\varepsilon}$, i.e., $\bigcap U_n \subseteq (fA)_{2\varepsilon}$, and our strategy is winning.
  Thus $(A)_{2\varepsilon}$ is co-meagre, as desired.
\end{proof}

\fi

Now, to our criterion.

\begin{dfn}
  Let $\cM$ be a Polish metric structure with universe $(M,d)$.
  We say that a (classical) countable structure $\cN$ is a \emph{countable approximating substructure} of $\cM$ if the following conditions are satisfied:
  \begin{itemize}
  \item The universe $N$ of $\cN$ is a dense countable subset of $(M,d)$.
  \item Any automorphism of $\cN$ extends to a (necessarily unique) automorphism of $\cM$, and $\Aut(\cN)$ is dense in $\Aut(\cM)$.
  \end{itemize}
  As usual, we view $\Aut(\cM)$ as a Polish topometric group by endowing it with the topology of point-wise convergence and the metric of uniform convergence.
  Similarly, we endow $\Aut(\cN)$ with the Polish group topology of point-wise convergence in the discrete set $N$, which refines the topology induced as a subset of $\Aut(\cM)$, and is, in general, strictly finer (for example, as soon as $\Aut(\cN) \subsetneq \Aut(\cM)$).

  We say that $\cN$ is a \emph{good countable approximating substructure} if in addition,
  \begin{itemize}
  \item For every open subset $U \subseteq \Aut(\cN)$ (in the topology of $\Aut(\cN)$!) and $\varepsilon > 0$, the set $(U)_\varepsilon$, as calculated in $\Aut(\cM)$, is open there.
  \end{itemize}
\end{dfn}

The following two are immediate.

\begin{lem}
  Let $(G,\tau,\partial)$ be a Polish topometric group.
  Then, for any open subset $U$ of $G$, $(U)_\varepsilon = (1)_\varepsilon {\cdot} U$ is open in $G$.
\end{lem}

\begin{lem}
  Let $\cM$ be a Polish metric structure, and $\cN$ be a countable approximating substructure.
  For any open subset $U$ of $\Aut(\cM)$, $U\cap \Aut(\cN)$ is open in $\Aut(\cN)$.
\end{lem}

And we conclude,

\begin{thm}
  \label{thm:ApproximationCriterion}
  Let $\cM$ be a Polish metric structure and let $\cN$ be a good countable approximating substructure of $\cM$.
  Then, for every co-meagre subset $A$ of $\Aut(\cN)^n$, the set $\overline{A}^\partial$, as calculated in $\Aut(\cM)^n$ is co-meagre there.
  In particular, if $\Aut(\cN)$ has ample generics then $\Aut(\cM)$ has ample generics as a topometric group.
\end{thm}
\begin{proof}
  Immediate from the observations above and \fref{thm:ApproximationCriterionAbstract}.
\end{proof}

\section{Examples}
\label{sec:Examples}

We can now provide examples of natural Polish groups with ample metric generics which are not closed subgroups of $\fS_\infty$, and which do not have ample generics.
We also provide one example of the opposite, namely a Polish group which has does have ample generics (and \textit{a fortiori} ample metric generics), even though $\partial_u$ is not discrete (so for this specific example, the considerations of the present paper just complicate things).

The first three examples are given as automorphism groups of familiar Polish structures: the Urysohn sphere $\bU_1$, the Hilbert space $\ell_2$ and the (measure algebra of) the standard probability space $[0,1]$.
Using \fref{thm:ApproximationCriterion} we prove for each such $\cM$ that $(G,\tau,\partial) = \Aut(\cM)$ has ample generics, where $\partial$ denotes the metric of uniform convergence, instead of the cumbersome $\partial_u^\cM$.
Since the intrinsic uniform distance $\partial_u$ of each such $G$ coarsens $\partial$, it follows in each case that $(G,\tau)$ has ample metric generics.
Moreover, in all three examples, conjugacy classes are known to be meagre: this is a theorem of Kechris for $\Iso(\bU_1)$ (see \cite{Rosendal:GenericElements}), a result variously attributed to Rokhlin or Del Junco for $\Aut([0,1],\lambda)$, and a theorem of Nadkarni for $U(\ell_2)$ (see \cite[Chapter~I.2]{Kechris:Ergodic} for both).

In fact, each of $\bU_1$, $\ell_2$ and the probability algebra associated to $([0,1],\lambda)$ is $\aleph_0$-categorical, meaning that in each case the action of $G$ on $\cM$ is approximately oligomorphic.
By \fref{fct:Aleph0CatUniformity} it follows that $\partial_u = \partial$, and in each case the same can be shown by a straightforward verification.
By \fref{thm:AmpleGenericsMetricBergman} (plus the fact that each group has ample metric generics, which we prove below) we obtain that in each case $(G,\partial)$ has the metric Bergman property.

The fourth and last example is $\fS_\infty^\bN$, and like all known examples of Polish groups with ample generics, embeds as a closed subgroup of $\fS_\infty$.

\subsection{The isometry group of the bounded Urysohn space}

Recall that the \emph{Urysohn space of diameter $1$} $\bU_1$ is the unique, up to isometry, metric space of diameter $1$ which is both \emph{universal} for Polish metric spaces of diameter at most $1$ and \emph{ultrahomogeneous}, i.e., isometries between finite subsets extend to isometries of the whole space.
For information about this space, we recommend consulting the volume \cite{TopAppl:Urysohn}.
Its isometry group is a Polish topometric group, when equipped with the topology of point-wise convergence and the metric of uniform convergence
$$\partial(g,h)= \sup \bigl\{ d(gx,hx) \colon x \in \bU_1 \bigr\}.$$
The countable counterpart of $\bU_1$ is the \emph{rational Urysohn space of diameter $1$} $\bQ\bU_1$, which is both ultrahomogeneous and universal for countable metric spaces with rational distances and diameter at most $1$.

The relationship between these two spaces is simple: $\bU_1$ is the completion of $\bQ\bU_1$.
In particular, $\bQ\bU_1$ is dense in $\bU_1$ and any isometry of $\bQ\bU_1$ extends uniquely to an isometry of $\bU_1$.
A theorem due to S.\ Solecki \cite{Solecki:ExtendingPartialIsometries}, combined with the techniques of Kechris and Rosendal, implies that the automorphism group of $\bQ\bU_1$ has ample generics (this is pointed out in \cite{Kechris-Rosendal:Turbulence}).
Finally, the classical techniques for building isometries of the Urysohn space (for example, a straightforward adaptation of \cite[Lemma~3]{Cameron-Vershik:UrysohnIsometryGroups}) show that whenever
$$U = \bigl\{ g \in \Iso(\bQ\bU_1) \colon ga_i = b_i \text{ for } i < n \bigr\}$$
is some basic open subset of $\Iso(\bQ\bU_1)$, one has
$$(U)_\varepsilon= \bigl\{ g \in \Iso(\bU_1) \colon d(ga_i,b_i) < \varepsilon \bigr\}.$$
Hence $(U)_\varepsilon$ is open in $\Iso(\bU_1)$ and all the conditions of \fref{thm:ApproximationCriterion} are satisfied, from which we may deduce that $(\Iso(\bU_1),\tau,\partial)$ has ample generics.
As mentioned earlier, it follows that $(\Iso(\bU_1),\partial)$ has the metric Bergman property.

Let $\bU$ denote the full (i.e., unbounded) Urysohn space.
We equip $\Iso(\bU)$ with the topology $\tau^\bU$ of point-wise convergence, as usual, as well as with the distance $\partial^\bU$ of uniform convergence, say truncated at one to ensure finiteness.
Then $(\Iso(\bU),\tau^\bU,\partial^\bU)$ is again a Polish topometric group, and the same arguments readily adapt to show that it has ample generics, so in particular $\Iso(\bU)$ has ample metric generics.
On the other hand, since $\Iso(\bU)$ acts transitively on $\bU$, it is clear that $(\Iso(\bU),\partial^\bU)$ does not have the metric Bergman property (since the action of $\Iso(\bU)$ on $\bU$ is \emph{not} approximately oligomorphic, \fref{thm:AmpleGenericsMetricBergman} does not apply).
Thus we obtain that $(\Iso(\bU),\partial^\bU)$ and $(\Iso(\bU_1),\partial)$ are not isomorphic topological groups.

A natural question is then: are $\Iso(\bU)$ and $\Iso(\bU_1)$ isomorphic as abstract groups?
We conjecture that $\Iso(\bU_1)$ is simple, which would imply that the two groups are not isomorphic, but do not know how to prove it.

\subsection{The unitary group of a separable Hilbert space}

The situation is a bit more complicated in $\mathcal{U}(\ell_2)$, the unitary group of the complex Hilbert space $\ell_2(\bN)$.
Since $\ell_2$ is unbounded, we consider instead its closed unit ball, equipped with functions $x \mapsto \alpha x$ for $|\alpha| \leq 1$ and $(x,y) \mapsto \half[x+y]$, from which $\ell_2$ can be recovered.
The automorphism group of the unit ball is again $U(\ell_2)$; the topology of point-wise convergence is the strong operator topology (or the weak one, since they agree on $U(\ell_2)$), and the distance of uniform convergence is $\partial(S,T)=\|S-T\|$.

Everything we need here has been worked out by Rosendal \cite{Rosendal:PropertyOB}.
Let ${\mathcal Q}$ denote the algebraic closure of $\bQ$, and consider the countable subset ${\mathcal Q}\ell_2$ of $\ell_2$ made up of the sequences with finite support and all coordinates in ${\mathcal Q}$.
The crucial fact here is that, since the norm of an element of ${\mathcal Q}\ell_2$ belongs to ${\mathcal Q}$, one can carry out the usual constructions of Hilbert space geometry (most importantly, Gram-Schmidt orthonormalisation) inside ${\mathcal Q}\ell_2$.
This is used in \cite{Rosendal:PropertyOB} to prove that the automorphism group of ${\mathcal Q}\ell_2$ has ample generics.

Using the Gram-Schmidt orthonormalisation procedure, it is not hard to prove that $(U)_\varepsilon$ is open for any basic open subset $U$ of $\mathcal U(\mathcal Q \ell_2)$ (we cannot, however, give a simple explicit formula as in the previous example).
Hence $\mathcal U(\ell_2)$, with its usual Polish topology, has ample metric generics.

\subsection{The group of measure-preserving bijections of $[0,1]$}
\label{sec:ExampleProbabilitySpace}

Denote by $\lambda$ the Lebesgue measure on the unit interval $[0,1]$.
We view its automorphism group $\Aut([0,1],\lambda)$ as the automorphism group of the Polish metric structure $(\text{MALG},0,1,\wedge,\vee,\neg,d)$, where $\text{MALG}$ denotes the measure algebra on $[0,1]$ and $d(A,B)=\lambda(A \Delta B)$ (see \cite{Kechris:Classical}).
The distance of uniform convergence is:
$$\partial(g,f)=\sup \bigl\{ \lambda (fA \Delta gA) \colon A \in \text{MALG} \bigr\}.$$
In ergodic theory, one often uses instead of this $\partial$ the distance $\partial'$ defined by
$$\partial'(g,f)=\lambda(\supp gf^{-1}).$$
It is well-known that these two distances are equivalent, indeed that $ \partial \le \partial' \le 2 \partial.$
It is also not difficult to check that $\partial'$ is lower semi-continuous (i.e., that if $\lambda(\supp g) > r$, then $\lambda(\supp f) > r$ for all $f$ in some neighbourhood of $g$).
We may therefore use $\partial'$ instead of $\partial$ whenever convenient.

This time, the approximating substructure is the countable measure algebra $\mathcal A$ generated by dyadic intervals.
Kechris and Rosendal proved in \cite{Kechris-Rosendal:Turbulence} that its automorphism group has ample generics; also, $\mathcal A$ is dense in $\text{MALG}$ and any automorphism of $\mathcal A$ extends to a measure-preserving automorphism of $[0,1]$.

To check the final condition of \fref{thm:ApproximationCriterion}, we use $\partial'$ instead of $\partial$.
It is enough to observe that, whenever $\mathcal B$ is a finite subalgebra of $\mathcal A$ with atoms $B_1,\ldots,B_n$, and $U$ is the open subset of $\Aut(\mathcal A)$ defined by
$$U=\{g \in \Aut(\mathcal A) \colon \forall B \in \mathcal B \ g(B)=B \}$$
then one has (where $(U)_\varepsilon$ is in the sense of $\partial'$ rather than of $\partial$)
$$(U)_\varepsilon= \left\{ g \in \Aut([0,1],\lambda) \colon \sum_{i=1}^n \lambda (B_i \setminus g(B_i)) < \varepsilon \right\}.$$

Thus $(U)_\varepsilon$ is indeed open in $\Aut([0,1],\lambda)$, and we see that $\Aut([0,1],\lambda)$ has ample metric generics.

It was pointed out to us by T.\ Tsankov that the techniques used by Kittrell and Tsankov \cite{Kittrell-Tsankov:TopologicalPropertiesOfFullGroups} to prove an automatic continuity theorem for the full group of an ergodic, countable, measure-preserving equivalence relation readily apply to show the following.
For the sake of completeness, a proof of this appears as \fref{apx:KT}.

\begin{fct}
  \label{fct:KT}
  Any morphism from $(\Aut([0,1],\lambda),\partial)$ into a separable topological group is continuous.
\end{fct}

Combining this with our result that $\Aut([0,1],\lambda)$ has ample metric generics, we obtain

\begin{thm}
  Any morphism from $\Aut([0,1],\lambda)$, endowed with its usual Polish topology, into a separable topological group is continuous.
\end{thm}

In addition, E.\ Glasner \cite{Glasner:AutMuRoelckePrecompact} has recently shown that the usual Polish topology of $\Aut([0,1],\lambda)$ is \emph{minimal}, i.e., $\Aut([0,1],\lambda)$ has no Hausdorff group topology strictly coarser than its usual Polish topology.
Combined with the automatic continuity property, this yields

\begin{thm}
  There is a unique separable Hausdorff group topology on $\Aut([0,1],\lambda)$ that is compatible with its algebraic structure, namely, its usual Polish topology.
\end{thm}

We should note here that Kallman \cite{Kallman:UniquenessForGroups} was the first to prove, in the eighties, that $\Aut([0,1],\lambda)$ has a unique Polish group topology compatible with its algebraic structure.

\subsection{The group $\fS_\infty^\bN$}

We conclude with an example of a different flavour, of a Polish group $(G,\tau)$ with ample generics, for which $\partial_u$ is not discrete.
In particular, $\partial_u$ need not be the finest distance for which ample generics exist: here $(G,\tau,\partial)$ has ample generics, with $\partial$ being the discrete metric, even though it is strictly finer than $\partial_u$.

The example we consider is $\fS_\infty^\bN$.
Indeed, we express $\bN$ as a disjoint union of infinitely many infinite sets $\bN = \bigcup_n A_n$, and obtain a homeomorphic embedding $\fS_\infty^\bN \hookrightarrow \fS_\infty$ by letting the $n$th copy of $\fS_\infty$ act on $A_n$.
The usual compatible left-invariant distance on $\fS_\infty$, $d_L(f,g) = 2^{- \min\{ n\colon f(n) \neq g(n) \}}$, induces one on $\fS_\infty$ via this embedding.
Then $\partial_u$ on $\fS_\infty^\bN$ is:
\begin{gather*}
  \partial_u(f,g) = \sup_{h \in \fS_\infty^\bN} d(fh,gh) = 2^{-k(f,g)}, \qquad \text{where} \qquad k(f,g) = \min \bigcup \{ A_n\colon f\rest_{A_n} \neq g\rest_{A_n} \}.
\end{gather*}
This distance is not discrete.
On the other hand, $\fS_\infty^\bN$ has ample generics since $\mathfrak{S}_{\infty}$ does, by the following.

\begin{prp}
  \label{prp:ProductAmpleGenerics}
  Let $(G_i)_{i \in I}$ be an at most countable family of Polish groups, and $G=\prod G_i$.
  Then $G$ has ample generics if and only if each $G_i$ has ample generics.

  This adapts to the topometric context: let $(G_i,\tau_i,\partial_i)$ be an at most countable family of Polish topometric groups, and $G=\prod G_i$ endowed with the product topology $\tau$ and the metric
  $$\partial\bigl( (g_i),(h_i) \bigr) = \sup \bigl\{ \partial_i(g_i,h_i) \colon i \in I \bigr\}.$$
  Then $(G,\tau,\partial)$ has ample generics iff $(G_i,\tau_i,\partial_i)$ has ample generics for all $i$.
\end{prp}
\begin{proof}
  We only give the proof of the first statement, since the topometric version is a straightforward adaptation.

  Note first that, since cooordinate projections induce surjective open morphisms from each $G^n$ onto $G_i^n$ which map diagonal conjugacy classes onto diagonal conjugacy classes, the existence of co-meagre diagonal conjugacy classes in $G^n$ immediately implies the existence of co-meagre diagonal conjugacy classes in each $G_i^n$.

  Now, assume that each $G_i$ has co-meagre diagonal conjugacy classes.
  For our purposes, it is enough to note the following easy fact: whenever $(X_i)_{i \in I}$ is an at most countable family of Polish spaces, and $(A_i)_{i \in I}$ are such that each $A_i$ is co-meagre in $X_i$, $\prod A_i$ is co-meagre in $\prod X_i$.
  Applying this with $X_i=G_i^n$ and $A_i$ equal to the co-meagre diagonal conjugacy class in $G_i^n$, we obtain the desired result.
\end{proof}

As an aside, let the discrete structure $\cM$ consist of $\bN$ equipped with predicates for the $A_n$.
Then $\fS_\infty^\bN = \Aut(\cM)$, and uniform convergence is discrete, providing an example where $\partial_u^\cM \neq \partial_u$.

\appendix

\section{Automatic continuity for $\Aut([0,1],\lambda)$}
\label{apx:KT}

We explain in this appendix how to use the techniques of \cite{Kittrell-Tsankov:TopologicalPropertiesOfFullGroups} to prove \fref{fct:KT}.
As before, $G = \Aut([0,1],\lambda)$, where $\lambda$ denotes the Lebesgue measure.
We equip it with the bi-invariant, lower semi-continuous distance $\partial$ (referred to as $\partial'$ in \fref{sec:ExampleProbabilitySpace})
$$\partial(g,f) = \lambda\left(\{x \colon gx \ne fx\} \right) = \lambda(\supp g^{-1}f).$$

We prove the following result, which implies \fref{fct:KT} (see \fref{dfn:Steinhaus} and the comments following it).

\begin{prp}
  $(G,\partial)$ is $38$-Steinhaus.
\end{prp}
\begin{proof}
  Let $W \subseteq $ be $\sigma$-syndetic, and let us prove that $W^{38}$ contains a neighbourhood of $1$.

  We may assume that $W$ is symmetric, and fix elements $g_n \in G$ such that $G = \bigcup_{n \in \bN} g_n W$.
  We also fix a measurable partition $(B_n)_{n \in \bN}$ of $[0,1]$, all of whose elements have strictly positive measure.
  For any measurable $B \subseteq \left[0,1\right]$, let $H_B = \{S \in G\colon \supp S \subseteq B\}$.

  \begin{clm}
    There exists $n \in \bN$ such that
    $$\forall T \in H_{B_n} \ \exists S \in W^2 \ S\rest_{B_n} = T\rest_{B_n}.$$
  \end{clm}
  \begin{clmprf}
    We first prove that there exists $n$ such that
    $$\forall T \in H_{B_n} \ \exists S \in g_nW \ S\rest_{B_n} = T\rest_{B_n}.$$
    If such were not the case, we could find for all $n$ some $T_n \in H_{B_n}$ such that no element of $g_nW$ coincides with $T_n$ on $B_n$.
    We may then glue all the $T_n$'s together to define $T \in G$ by setting $T(x) = T_n(x)$ for all $x \in B_n$.
    Such a $T$ would not belong to any $g_nW$, contradicting the definition of the sequence $(g_n)$.

    Since $\id \in H_{B_n}$ and $(g_n W)^{-1} (g_n W) = W^2$, our claim follows.
  \end{clmprf}

  We fix $n$ as in the first claim and let $B=B_n$.

  \begin{clm}
    The set $W^2$ contains a non trivial involution $S$ whose support is contained in $B$ and has measure less that $\lambda(B)/2$.
  \end{clm}
  \begin{clmprf}
    Indeed, pick any non trivial involution $T \in H_B$ whose support has measure less than $\lambda(B)/4$, and consider the subgroup $\Gamma$ of $H_B$ made up of all elements $U$ such that
    $$\forall b \in B \ U(b) \in \{b,T(b)\}.$$
    The subgroup $\Gamma$ is uncountable, so there must exist $m$ such that $g_mW$ contains two distinct elements $U,V$ of $\Gamma$.
    Then $S = UV = U^{-1}V \in W^2$ is an involution and $\lambda( \supp S) \leq \lambda( \supp U ) + \lambda(\supp V ) \leq \lambda(B)/2$, as desired.
  \end{clmprf}

  We fix such an $S$, as well as $C \subseteq B$ such that $\supp S \subseteq C$ and $\lambda(C)=2 \lambda( \supp S )$.

  \begin{clm}
    We have $H_C \subseteq W^{36}$.
  \end{clm}
  \begin{clmprf}
    We first note that any involution $T$ in $H_C$ whose support has the same measure as that of $S$ is conjugate to $S$ in $H_C$, so by our first claim we may find $V \in W^2$ such that $V(C)=C$ and $(VSV^{-1})\rest_C = T\rest_C$.
    Since $S$ and $T$ are equal to the identity outside $H_C$, this actually implies that $VSV^{-1}=T$, and so any involution of $H_C$ whose support has the same measure as that of $S$ belongs to $W^6$.

    It is easy to see that any involution of $H_C$ is the product of two involutions of $H_C$ whose support has the same measure as that of $S$; hence any involution in $H_C$ belongs to $W^{12}$.
    Since Ryzhikov \cite{Ryzhikov:Representations} proved that any element of $H_C$ is the product of at most $3$ involutions, we are done.
  \end{clmprf}

  We are now ready to conclude the proof: consider a sequence $(T_n)$ of elements of $G$ such that $\partial(T_n,1) \to 0$.
  Let $D_n = \supp T_n$ and $D = \bigcup_m g_m D_m$.
  It will be enough to show that there exists $m$ such that $T_m \in W^{38}$, and without loss of generality (going to a sub-sequence if necessary) we may assume that $\sum \lambda(D_n) \le \lambda(C)$, hence $\lambda(D) \le \lambda(C)$.

  Then we may find $A \subseteq C$ such that $\lambda(A) = \lambda(D)$.
  Pick some $S \in G$ such that $S(A)=D$; there exists $m$ such that $S \in g_mW$, and we let $T=g_m^{-1}S \in W$.
  We have $T(A) \supseteq D_m$, hence $T^{-1}T_mT \in H_A \subseteq H_C \subseteq W^{36}$.
  Then $T_m$ belongs to $WW^{36}W=W^{38}$, and we are done.
\end{proof}

\if\mymode i
\bibliographystyle{begnac}
\else
\bibliographystyle{alpha}
\fi
\bibliography{begnac}

\end{document}